\documentclass[12pt]{article}
\usepackage[margin=1in]{geometry}
\DeclareMathAlphabet\mathbfcal{OMS}{cmsy}{b}{n}
\usepackage[utf8]{inputenc}
\usepackage[dvipsnames]{xcolor}
\usepackage{enumerate}
\usepackage[utf8]{inputenc}
\usepackage{amsfonts}
\usepackage{amsthm}
\usepackage{amsmath}
\usepackage{lscape}
\usepackage{mathrsfs}
\usepackage{graphicx}
\usepackage{tikz-cd}
\usepackage{calc}
\usepackage{float}
\usepackage{stmaryrd}
\usepackage{multirow}
\usepackage{lipsum}
\usepackage{amssymb}
\usepackage{mathrsfs}
\usepackage{mathtools}
\usepackage{mathdots}
\usepackage{fancyhdr}
\usepackage{tikz}
\usepackage{faktor}
\usepackage{setspace}
\usetikzlibrary{decorations.markings}
\newtheorem{theorem}{Theorem}
\theoremstyle{plain}
\usepackage{sectsty}
\usepackage{kpfonts}
\newtheorem{corollary}{Corollary}
\newtheorem{definition}{Definition}

\newtheorem{lemma}{Lemma}

\newtheorem{remark}{Remark}

\newtheorem{question}{Question}

\newtheorem*{proposition*}{Proposition}

\newcommand{\GL}{{\mbox{GL}}}
\newcommand{\Hom}{{\mbox{Hom}}}

\newcommand{\End}{{\mbox{End}}}

\newcommand{\vol}{{\mbox{vol}}}

\newtheorem*{observation*}{Observation}
\newtheorem*{theorem*}{Theorem}
\newtheorem*{claim*}{Claim}  
\usepackage{blindtext}
\usepackage{fancyhdr}
\title{Complex structures as critical points}
\author{Gabriella Clemente}
\date{}
\allowdisplaybreaks
\makeatletter
\renewcommand\tableofcontents{%
    \@starttoc{toc}%
}
\begin{document}
\maketitle

\begin{abstract}
This note aims at obtaining a variational characterization of complex structures by means of a calculus of variations for real vector bundle valued differential forms, and outlines a perspective to study existence questions via functionals and stability notions.
\end{abstract} 

\section*{Introduction}
It is believed that the $6$ dimensional sphere $S^6$ provides a positive answer to the\\

\begin{question}\label{QF1}
Are there any (compact) almost-complex manifolds of real dimension at least $6$ that do not support a complex structure?\\
\end{question}

This note is concerned with the moduli space version of this question, which will be referred to as the\\

\noindent
\textbf{Complex structure existence (CSE) problem:} Obtain a full characterization of almost-complex, but not complex manifolds.\\

As it often happens in geometry, when studying the existence of structures in a given class, one must deal with the difficulty that the class is infinite dimensional. This is the case for the space of almost-complex structures, which is a Fr{\'e}chet manifold. Deciding if an almost-complex manifold is complex requires, in theory, computing what is known as the Nijenhuis tensor infinitely many times, once per each almost-complex structure. So the CSE problem is mostly inaccessible through the Nijenhuis tensor directly, even when the computational complexity of the problem is greatly reduced by means of the topology of the space of almost-complex structures. To have an idea of what this means, consider $S^6,$ which has a unique, up to homotopy, almost-complex structure, and this still does not seem to reduce the problem to something reasonable. While the CSE is a moduli level problem, an answer to Question \ref{QF1} lies in obstruction theory, and curvature presents an obstruction.

The Nijenhuis tensor of an almost-complex structure can be rewritten in terms of the covariant exterior derivative with respect to an arbitrary torsion-free connection on the tangent bundle (cf.\ Lemma 1, \cite{ACOTI}). Specific choices of such connections lead to different geometric consequences. Higher derivatives of the Nijenhuis tensor involve the curvature of the connection. So these higher order equations act as a bridge between the metric geometry of the manifold, and its almost-complex, and complex geometries. By choosing the Levi-Civita connection of a Riemannian metric, one can examine how the various notions of curvature (e.g.\ $Rm,$ $Ric,$ $sec,$ $s,$ etc.) interact with hypothetical complex structures. Based on the curvature obstruction equations obtained in \cite{ACOTI}, it is reasonable to expect that canonical metrics can prevent the existence of complex structures in the compact case. The curvature point of view allows for a narrowing down of Question \ref{QF1}. One may ask, for instance, if Riemannian metrics of positive constant curvature obstruct the existence of complex structures on compact manifolds of real dimension at least $4.$ If the answer is yes, then $S^6$ could not be complex. A positive answer would also imply that compact complex manifolds in high enough dimensions are at the very least ``ridged," meaning that there needs to be a minimal amount of irregularity in the curvature of any Riemannian metric, in this way ruling out the positive constant curvature scenario. The CSE problem will be treated here in a formal way, however, as the priority is to set up a language of vector bundle valued forms on real manifolds and their calculus of variations.

This note begins with the algebraic setup for said calculus, proceeds to study the variational character of (some nearly) integrable almost-complex structures, and ends with a perspective for advancing the CSE problem from the viewpoint of K-stability. 

\section{Algebraic machinery}
Here $M$ is a compact, almost-complex manifold of real even dimension $n \geq 2,$ and \[\Omega^{\bullet} (M, T_M)=\bigoplus_{k=0}^n \Omega^k (M, T_M)\] is the space of tangent bundle valued differential forms. A form $\gamma \in \Omega^{\bullet} (M, T_M)$ decomposes as $\gamma=\sum_{k=0}^n \gamma_k,$ $\gamma_k \in  \Omega^k (M, T_M).$ The projection onto $k$-th degree forms will be denoted by $p_k$; i.e.\ \[p_k:\Omega^{\bullet} (M, T_M) \to \Omega^k (M, T_M), \quad p_k (\gamma)=\gamma_k.\] The material on actions on the space of tangent bundle valued forms that comes next is from \cite{ACOTI}, and will be summarized here for the sake of completeness. 

The spaces \[\Omega^{\bullet} \big(M, \End_{\mathbb{R}}(T_M)\big)=\bigoplus_{k=0}^n \Omega^k \big(M, \End_{\mathbb{R}}(T_M)\big), \mbox{ and } \Omega^{\bullet} \big(M, \bigwedge^{\bullet} {T_M}\big)=\bigoplus_{k=0}^{2n} \bigoplus_{p+q=k}  \Omega^p \big(M, \bigwedge^q {T_M}\big)\] of endomorphism, respectively polyvector valued forms can be equipped with operations defined as follows. 

Consider the action \[\cdot:\End_{\mathbb{R}}(T_M) \times T^*_M \otimes T_M \to T^*_M \otimes T_M,\] \[(S, f\otimes e)\mapsto S\cdot (f\otimes e):=f\otimes S(e).\] The product in $\Omega^{\bullet} \big(M, \End_{\mathbb{R}}(T_M)\big)$ is given as follows: for any $\alpha \in \Omega^k \big(M, \End_{\mathbb{R}}(T_M)\big), \beta \in \Omega^l \big(M, \End_{\mathbb{R}}(T_M)\big),$ $\alpha \wedge \beta \in \Omega^{k+l} \big(M, \End_{\mathbb{R}}(T_M)\big),$ and 
\begin{equation*}
\begin{split}
(\alpha \wedge \beta)(X_1,\dots,X_{k+l})=\frac{1}{k!l!}\sum_{\sigma \in S_{k+l}} sign(\sigma) \alpha(X_{\sigma(1)},\dots,X_{\sigma(k)})\cdot \beta(X_{\sigma(k+1)},\dots,X_{\sigma(k+l)}),
\end{split}
\end{equation*}
where indeed \[\alpha(X_{\sigma(1)},\dots,X_{\sigma(k)}), \beta(X_{\sigma(k+1)},\dots,X_{\sigma(k+l)}) \in \End_{\mathbb{R}}(T_M).\]

The product in $\Omega^{\bullet} \big(M, \bigwedge^{\bullet} {T_M}\big)$ is defined with a normalization factor{\footnote{The $\frac{1}{2}$ factor is there to ensure the validity of Lemma \ref{K}.}, but without any twisting. Namely, for any $\gamma \in  \Omega^i \big(M, \bigwedge^j {T_M}\big), \theta \in \Omega^{k} \big(M, \bigwedge^{l} {T_M}\big),$ $\gamma \wedge \theta \in \Omega^{i+k} \big(M, \bigwedge^{j+l} {T_M}\big),$ and

\begin{equation*}
\begin{split}
(\gamma \wedge \theta)(X_1,\dots,X_{i+k})=\frac{1}{2}\frac{1}{i!k!}\sum_{\sigma \in S_{i+k}} sign(\sigma) \gamma(X_{\sigma(1)},\dots,X_{\sigma(i)})\wedge \theta(X_{\sigma(i+1)},\dots,X_{\sigma(i+k)}),
\end{split}
\end{equation*}
where \[\gamma(X_{\sigma(1)},\dots,X_{\sigma(i)})\in \bigwedge^j {T_M}, \mbox{ and }\theta(X_{\sigma(i+1)},\dots,X_{\sigma(i+k)})\in \bigwedge^{l} {T_M}.\] 

Thus, $\Omega^{\bullet} \big(M, \End_{\mathbb{R}}(T_M)\big),$ and $\Omega^{\bullet} \big(M, \bigwedge^{\bullet} {T_M}\big)$ can be regarded as graded algebras with these products.

Note that if $\gamma \in \Omega^i(M,T_M),$ and $\theta \in \Omega^k(M,T_M),$ then $\gamma \wedge \theta = (-1)^{ik+1} \theta \wedge \gamma$ so that $\gamma, \theta$ anti-commute iff one of $i$ and $k$ is even. In particular, when $\gamma$ is homogeneous of even degree, $\gamma \wedge \gamma=0.$

The tangent bundle forms $\Omega^{\bullet} (M, T_M)$ are acted on the left by $\Omega^{\bullet} \big(M, \End_{\mathbb{R}}(T_M)\big),$ and on the right by $\Omega^{\bullet} \big(M, \bigwedge^{\bullet} {T_M}\big).$ The left action combines the evaluation map \[ \End_{\mathbb{R}}(T_M) \times T_M \to T_M, \quad (S,y) \mapsto S(y),\] with the wedge product of forms: for any $\rho \in \Omega^s (M, T_M),$ and $\alpha \in \Omega^k \big(M, \End_{\mathbb{R}}(T_M)\big),$ $\alpha \wedge \rho \in \Omega^{k+s}(M, T_M),$ and 

\begin{equation*}
\begin{split}
(\alpha \wedge \rho)(X_1,\dots,X_{k+s})=\frac{1}{k!s!}\sum_{\sigma \in S_{k+s}} sign(\sigma) \alpha(X_{\sigma(1)},\dots,X_{\sigma(k)})\big(\rho(X_{\sigma(k+1)},\dots,X_{\sigma(k+s)})\big),
\end{split}
\end{equation*}
where \[\alpha(X_{\sigma(1)},\dots,X_{\sigma(k)})\in \End_{\mathbb{R}}(T_M), \mbox{ and }\rho(X_{\sigma(k+1)},\dots,X_{\sigma(k+s)}) \in T_M.\]
The right action makes use of the isomorphism \[\Hom_{\mathbb{R}}\Big(\bigwedge^j {T_M},T_M\Big)\simeq \bigwedge^j {T^*_M} \otimes T_M,\] the evaluation map 
\[\Hom_{\mathbb{R}}\Big(\bigwedge^j {T_M},T_M\Big) \times \bigwedge^j {T_M} \to T_M,\]
\[(C:=c\otimes z,x)\mapsto C(x)=(c\otimes z)(x)=c(x)z,\] and the wedge product of $\Omega^{\bullet}(M).$

If $\gamma \in  \Omega^i \big(M, \bigwedge^j {T_M}\big),$ and $s \geq j,$ set 

\begin{equation*}
\begin{split}
(\rho \wedge \gamma)(X_1,\dots,X_{s-j+i})=\frac{1}{(s-j)!i!}\sum_{\sigma \in S_{s-j+i}} sign(\sigma) \rho(X_{\sigma(1)},\dots,X_{\sigma(s-j)},\cdot,\dots,\cdot)\big(\gamma(X_{\sigma(s-j+1)},\dots,X_{\sigma(s-j+i)})\big),
\end{split}
\end{equation*}
where indeed \[\rho(X_{\sigma(1)},\dots,X_{\sigma(s-j)},\cdot,\dots,\cdot) \in \Hom_{\mathbb{R}}\Big(\bigwedge^j {T_M},T_M\Big), \mbox{ and }\gamma(X_{\sigma(s-j+1)},\dots,X_{\sigma(s-j+i)}) \in \bigwedge^j {T_M};\] if $s<j,$ declare $\rho \wedge \gamma=0.$

For simplicity, the same notation is being used to express all products at play here. However, parentheses and context do away with ambiguities. 

Ordinary differential forms also act on $\Omega^{\bullet} (M, T_M).$ If $\beta \in \Omega^k(M),$ and $B \in \Omega^l(M,T_M),$ then 

\begin{equation*}
(\beta \wedge B)(X_1,\dots,X_{k+l})=\frac{1}{k! l!} \sum_{\sigma \in S_{k+l}} sign(\sigma) \beta(X_{\sigma(1)},\dots,X_{\sigma(k)})B(X_{\sigma(k+1)},\dots,X_{\sigma(k+l)}).
\end{equation*}

\begin{remark}\label{KC}
Indeed, $\Omega^{\bullet} (M, T_M)$ is both a left $\Omega^{\bullet} (M)$-module, and a left $\Omega^{\bullet} \big(M, \End_{\mathbb{R}}(T_M)\big)$-module. The right $\Omega^{\bullet} \big(M, \bigwedge^{\bullet} {T_M}\big)$-action on $\Omega^{\bullet} (M, T_M)$ satisfies all module axioms, but one, namely, $\rho \wedge (\gamma \wedge \gamma') \neq (\rho \wedge \gamma) \wedge \gamma'.$
\end{remark}

\section{Integration}
A more thorough treatment of integration of bundle forms can be found in \cite{GSVO1}. Let $g$ be a Riemannian metric on $M.$ This induces a metric on $\bigwedge^p {T_M},$ defined as \[g(v_1 \wedge \dots \wedge v_p, w_1 \wedge \dots \wedge w_p)=\det{\begin{pmatrix} g(v_i,w_j) \end{pmatrix}_{1 \leq i, j \leq p}},\] where $v_i, w_i \in T_M.$ 

For any $p\geq 1,$ the space $\Omega^{\bullet} \big(M, \bigwedge^p{T_M}\big)$ can be endowed with an $L^2$-inner product in the following way. 

Let $(x_i)_{i=1}^n$ be local coordinates on $M,$ which give a local frame $\big(\frac{\partial}{\partial x_i}\big)_{i=1}^n$ of $T_M.$ Then, $\bigwedge^p {T_M}$ is locally framed by $\frac{\partial}{\partial x_{i_1}} \wedge \dots \wedge \frac{\partial}{\partial x_{i_p}},$ $1 \leq i_1 < \dots <i_p \leq n.$ First, consider the bilinear map \[\wedge_g:\Omega^k \big(M, \bigwedge^p{T_M}\big) \otimes \Omega^l \big(M, \bigwedge^p{T_M}\big) \to \Omega^{k+l}(M),\] where if \[\alpha=\sum a_{i_1 \dots i_p} \otimes \frac{\partial}{\partial x_{i_1}} \wedge \dots \wedge \frac{\partial}{\partial x_{i_p}} \in \Omega^k \big(M, \bigwedge^p{T_M}\big),\] and if \[\beta=\sum b_{j_1 \dots j_p} \otimes \frac{\partial}{\partial x_{j_1}} \wedge \dots \wedge \frac{\partial}{\partial x_{j_p}} \in \Omega^l \big(M, \bigwedge^p{T_M}\big),\] \[\alpha \wedge_g \beta = \sum a_{i_1 \dots i_p} \wedge b_{j_1 \dots j_p} g\big(\frac{\partial}{\partial x_{i_1}} \wedge \dots \wedge \frac{\partial}{\partial x_{i_p}}, \frac{\partial}{\partial x_{j_1}} \wedge \dots \wedge \frac{\partial}{\partial x_{j_p}}\big).\] Next, consider the pairing \[\langle \cdot,\cdot \rangle_g:\Omega^k \big(M, \bigwedge^p{T_M}\big) \otimes \Omega^k \big(M, \bigwedge^p{T_M}\big) \to \mathbb{R},\] where if \[\sigma=\sum s_{j_1 \dots j_p} \otimes \frac{\partial}{\partial x_{j_1}} \wedge \dots \wedge \frac{\partial}{\partial x_{j_p}} \in \Omega^k \big(M, \bigwedge^p{T_M}\big),\] then \[\langle \alpha, \sigma \rangle_g=\sum \langle a_{i_1 \dots i_p},s_{j_1 \dots j_p}\rangle g\big(\frac{\partial}{\partial x_{i_1}} \wedge \dots \wedge \frac{\partial}{\partial x_{i_p}}, \frac{\partial}{\partial x_{j_1}} \wedge \dots \wedge \frac{\partial}{\partial x_{j_p}}\big).\] An extension of the Hodge star operator to $\Omega^{\bullet} \big(M, \bigwedge^p{T_M}\big)$ is possible by requiring it to have the defining property that \[\star_g:\Omega^k \big(M, \bigwedge^p{T_M}\big) \to \Omega^{n-k} \big(M, \bigwedge^p{T_M}\big),\] \[\alpha \wedge_g \star_g \alpha'=\langle \alpha,\alpha'\rangle_g \vol_g.\] With the pairing \[\langle \cdot, \cdot \rangle_k: \Omega^k \big(M, \bigwedge^p{T_M}\big) \otimes \Omega^k \big(M, \bigwedge^p{T_M}\big) \to \mathbb{R},\] \[\langle \alpha,\alpha' \rangle_k=\int_M \alpha \wedge_g \star_g \alpha',\] the $L^2$-inner product is \[\langle\langle \cdot,\cdot \rangle\rangle:\Omega^{\bullet} \big(M, \bigwedge^p{T_M}\big) \otimes \Omega^{\bullet} \big(M, \bigwedge^p{T_M}\big) \to \mathbb{R},\] \[\langle\langle A,B\rangle\rangle:=\sum_{k \geq 0} \langle A_k,B_k\rangle_k,\] where $A=\sum_{k \geq 0} A_k, B=\sum_{k \geq 0} B_k,$ and $A_k, B_k \in \Omega^k \big(M, \bigwedge^p{T_M}\big).$ For $p=1,$ this specifies to an $L^2$-inner product on $\Omega^{\bullet} (M, T_M),$ and in ths instance $g$ is the actual metric on $T_M.$

\section{Augmented integrability and complex structures}
Let the space of almost-complex structures on $M$ be denoted by \[AC(M):=\{A \in \Omega^1 (M, T_M) \mid A \circ A=-Id\}.\] The manifold $M$ is complex if it carries an $A \in AC(M)$ such that the Nijenhuis tensor of $A,$ \[N_A (\zeta, \eta)=[A(\zeta), A(\eta)]-A([A(\zeta),\eta]+[\zeta,A(\eta)])-[\zeta,\eta],\] vanishes for all vector fields $\zeta, \eta \in \mathfrak{X}(M)$ \cite{New}. In this case, $A$ is called an integrable almost-complex structure or a complex structure. Throughout, $\nabla$ will be assumed to be a symmetric connection on $T_M.$ 

Consider $d^{\nabla},$ the covariant exterior derivative associated to $\nabla,$ which at degree $k,$ is the map $d^{\nabla}:\Omega^k (M, T_M) \to \Omega^{k+1} (M, T_M),$ 

\begin{equation*}
\begin{split}
(d^{\nabla} \alpha)(\zeta_0,\dots,\zeta_k)&=\sum_{i=0}^k (-1)^i \nabla_{\zeta_i} \alpha(\zeta_0,\dots,\widehat{\zeta_i},\dots,\zeta_k)+\\
&\sum_{0 \leq i <j \leq k} (-1)^{i+j} \alpha ([\zeta_i,\zeta_j],\dots,\widehat{\zeta_i},\dots,\widehat{\zeta_j},\dots,\zeta_k)\\
&=\sum_{i=0}^k (-1)^i (\nabla_{\zeta_i} \alpha) (\zeta_0,\dots,\hat{\zeta_i},\dots,\zeta_k).
\end{split}
\end{equation*}
Recall that $d^{\nabla}$ may be defined too as the unique linear operator that satisfies $d^{\nabla} \psi = \nabla \psi$ if $\psi \in \Omega^0 (M,T_M),$ and the graded Leibniz rule \[d^{\nabla}(\beta \wedge B)=d \beta \wedge B +(-1)^k \beta \wedge d^{\nabla} B\] for $\beta \in \Omega^k(M),$ $B \in \Omega^l(M,T_M).$ In the context of the previous section, this has a formal adjoint, $\delta^{\nabla}.$

The integrability condition has an alternative expression in the language of bundle-valued forms.
 
\begin{lemma}{(Lemma 1, \cite{ACOTI})}\label{K}
$A \in AC(M)$ is integrable iff \[d^{\nabla} A \wedge (A \wedge A)-d^{\nabla} A=0.\]
\end{lemma}

In addition to integrable almost-complex structures, one may consider nearly integrable ones as described below. These structures could provide a new lens through which to look at the CSE problem. 

Let $\rho \in \Omega^k (M, T_M),$ and $\alpha \in \Omega^1(M)$ be any non-trivial, $d$-closed form. Consider the quantities \[I^{\nabla}_{\rho}:=d^{\nabla} \rho \wedge (\rho \wedge \rho)-d^{\nabla} \rho,\] \[I^{\nabla}_{\alpha, \rho}=\alpha \wedge I^{\nabla}_{\rho},\] and \[I^{\alpha, \nabla}_{\rho}:=(\alpha \wedge d^{\nabla} \rho) \wedge (\rho \wedge \rho)-\alpha \wedge  d^{\nabla} \rho.\] To be clear, in contrast with $I^{\alpha, \nabla}_{\rho},$ where $\rho \wedge \rho \in \Omega^{2k} \big(M, \bigwedge^{2} {T_M}\big)$ is acting on the right of $\alpha \wedge d^{\nabla} \rho \in \Omega^{k+2}(M,T_M),$ in $I^{\nabla}_{\alpha, \rho},$ $\alpha$ is acting on the left of $d^{\nabla} \rho \wedge (\rho \wedge \rho) \in \Omega^{3k-1}(M,T_M),$ and the $3k$-forms $(\alpha \wedge d^{\nabla} \rho) \wedge (\rho \wedge \rho)$ and $\alpha \wedge \big(d^{\nabla} \rho \wedge (\rho \wedge \rho)\big),$ hence also the quantities $I^{\alpha, \nabla}_{\rho}$ and $I^{\nabla}_{\alpha, \rho},$ are different.

Notice that for an almost-complex structure $A \in \Omega^1(M,T_M),$ $I^{\nabla}_A$ is the integrability form of Lemma \ref{K}, and $I^{\nabla}_{\alpha, A}$ slightly generalizes complex structures, while \[I^{\alpha, \nabla}_A= (\alpha \wedge d^{\nabla} A) \wedge (A \wedge A)-\alpha \wedge d^{\nabla} A\] augments $I^{\nabla}_A$ to a $3$-form that though inexact, appears to be the next best thing to a $d^{\nabla}$-exact form encoding integrability. What is meant here is that although $d^{\nabla} \big(A \wedge (A \wedge A)\big)=0$ because $A \wedge (A \wedge A)=0,$ meaning that the most obvious ansatz $A \wedge (A \wedge A)-A$ fails to be a primitive of $I^{\nabla}_A,$ $(\alpha \wedge A)\wedge (A\wedge A)-\alpha \wedge A$ is perhaps the next, most naive guess, and it is not too far from being exact as \[d^{\nabla} \big((\alpha \wedge A)\wedge (A\wedge A)-\alpha \wedge A\big)=-I^{\alpha, \nabla}_A+2(\alpha \wedge A) \wedge (d^{\nabla} A \wedge A).\] Exactness has an impact on the calculus of variations technique that is employed in the proof of Theorem \ref{EP}, for example, and can overall simplify calculations. More precisely, exactness can soften the requirements imposed on the domain of the relevant functionals. 

In summary, the main reason for considering the augmented integrability condition $I^{\nabla}_{\alpha,A}=0$ is geometrical, and that for looking at $I^{\alpha, \nabla}_A=0$ is more technical (exactness). Before proceeding with a comparison of these classes of almost-complex structures, it is useful to note that

\begin{lemma}{(A graded product rule)}\label{LRT}
If $\alpha \in \Omega^k(M,T_M),$ $\beta \in  \Omega^l(M,T_M),$ and if $\nabla'$ denotes the connection on $\bigwedge^{2} {T_M}$ that is induced by the torsion-free connection $\nabla$ on $T_M,$ then
\[d^{\nabla'}(\alpha \wedge \beta)=d^{\nabla} \alpha \wedge \beta+(-1)^k \alpha \wedge d^{\nabla} \beta.\]
\end{lemma}

\begin{proof}
For simplicity, assume that $\alpha=a \otimes v,$ and $\beta=b\otimes w$ are pure tensors. In local coordindates $(x_i)_{i=1}^n$ on $M,$ one has that 
\begin{equation*}
\begin{split}
d^{\nabla'} (\alpha \wedge \beta)&=d^{\nabla'} (a\wedge b \otimes v \wedge w)\\
&=d(a \wedge b) \otimes v \wedge w+(-1)^{k+l} a\wedge b \wedge \nabla' v \wedge w\\
&=da \wedge b\otimes v \wedge w+(-1)^k a\wedge db \otimes v\wedge w+(-1)^{k+l} \sum_{i=1}^n (a \wedge b \wedge dx_i) \otimes \nabla'_{\frac{\partial}{\partial x_i}} v\wedge w\\
&=da \wedge b\otimes v \wedge w+(-1)^k a\wedge db \otimes v\wedge w+\\
&(-1)^{k+l} \sum_{i=1}^n (a \wedge b \wedge dx_i) \otimes \big(\nabla_{\frac{\partial}{\partial x_i}} v\wedge w+v\wedge \nabla_{\frac{\partial}{\partial x_i}} w \\
&=\Big(da \wedge b\otimes v \wedge w+(-1)^{k+l} \sum_{i=1}^n (a \wedge b \wedge dx_i) \otimes \nabla_{\frac{\partial}{\partial x_i}} v\wedge w\Big)+\\
&(-1)^k\Big(a\wedge db \otimes v\wedge w+(-1)^l \sum_{i=1}^n (a \wedge b \wedge dx_i) \otimes v\wedge \nabla_{\frac{\partial}{\partial x_i}} w\\
&=d^{\nabla} \alpha \wedge \beta+(-1)^k \alpha \wedge d^{\nabla} \beta.
\end{split}
\end{equation*}
\end{proof}

So by Lemma \ref{LRT}, in general, for any $\rho \in \Omega^k(M,T_M),$ 
\begin{equation*}
\begin{split}
d^{\nabla} \big((\alpha \wedge \rho) \wedge (\rho \wedge \rho)-\alpha \wedge \rho\big)&=-(\alpha \wedge d^{\nabla} \rho) \wedge (\rho \wedge \rho)+\big(1+(-1)^{k+1}\big) (\alpha \wedge \rho) \wedge (d^{\nabla} \rho \wedge \rho)+\\
&\alpha \wedge d^{\nabla} \rho \\
&=-I^{\alpha, \nabla}_{\rho}+\big(1+(-1)^{k+1}\big)(\alpha \wedge \rho) \wedge (d^{\nabla} \rho \wedge \rho) \\
&=\begin{cases}
-I^{\alpha, \nabla}_{\rho} & if k\mbox{ is even}\\
-I^{\alpha, \nabla}_{\rho}+2(\alpha \wedge \rho) \wedge (d^{\nabla} \rho \wedge \rho) & if k\mbox{ is odd}.
\end{cases}
\end{split}
\end{equation*}
Moreover, since $\rho \wedge \rho=0$ if $k$ is even, in which case, $I^{\alpha, \nabla}_{\rho}$ is $d^{\nabla}$-exact (i.e.\ $I^{\alpha, \nabla}_{\rho}=d^{\nabla}(\alpha \wedge \rho)$), the above simplifies to
\begin{equation*}
\begin{split}
d^{\nabla} \big((\alpha \wedge \rho) \wedge (\rho \wedge \rho)-\alpha \wedge \rho\big)&=\begin{cases}
\alpha \wedge d^{\nabla} \rho & if k\mbox{ is even}\\
-I^{\alpha, \nabla}_{\rho}+2(\alpha \wedge \rho) \wedge (d^{\nabla} \rho \wedge \rho) & if k\mbox{ is odd}.
\end{cases}
\end{split}
\end{equation*}

This brief discussion is supposed to strengthen the intuition for what it means for $I^{\alpha, \nabla}_A,$ $A \in AC(M),$ to be not too far from being exact.

\begin{definition}\label{alpa}
Those $A \in AC(M)$ such that $I^{\alpha, \nabla}_A=0$ will be called \emph{quasi}-$\alpha$-\emph{integrable}, and those satisfying $I^{\nabla}_{\alpha, A}=0$ will be called $\alpha$-integrable. 
\end{definition}

The concepts of (quasi-)$\alpha$-integrability, integrability, and $d^{\nabla}$-closedness are related to each other in the following ways (see Figures \ref{PR1} and \ref{PR2} below as well). Let $A \in AC(M).$ In order to be consistent with \cite{ACOTI}, $A$ will be called \emph{special} if $d^{\nabla} A=0,$ and $\alpha$-special if $\alpha \wedge d^{\nabla} A=0.$ First of all, if $A$ is special, it is $\alpha$-special, quasi-$\alpha$-integrable, integrable, and $\alpha$-integrable for all auxiliary $\alpha \in \Omega^1(M)$ ($d^{\nabla} A=0 \implies \alpha \wedge d^{\nabla} A=0, I^{\alpha, \nabla}_A=0, I^{\nabla}_A=0, \mbox{ and }I^{\nabla}_{\alpha, A}=0$). It should also be evident that if $A$ is integrable, it is $\alpha$-integrable, for all auxiliary $\alpha$ ($I^{\nabla}_A = 0 \implies I^{\nabla}_{\alpha, A}=0$). Observe that 

\begin{equation}\label{onethird} 
\alpha \wedge \big(d^{\nabla} A \wedge (A \wedge A)\big)=\frac{1}{3}(\alpha \wedge d^{\nabla} A) \wedge (A \wedge A),
\end{equation} which can be seen straight from the definitions in the first section:
\begin{equation*}
\begin{split}
\big((\alpha \wedge d^{\nabla} A)\wedge (A \wedge A)\big)(X_1,X_2,X_3)&=\frac{1}{2}\sum_{\sigma \in S_3} sign(\sigma) \big[\frac{1}{2} \sum_{\theta \in S_3} sign(\theta) \alpha(X_{\theta(\sigma(1))}) \times \\
&d^{\nabla} A \big(A(X_{\theta(\sigma(2))}), A(X_{\theta(\sigma(3))})\big) \big]\\
&=3\big[\alpha(X_1)d^{\nabla} A(A(X_2),A(X_3))-\alpha(X_2)d^{\nabla} A(A(X_1),A(X_3))\\
&-\alpha(X_3)d^{\nabla} A(A(X_2),A(X_1))\big],
\end{split}
\end{equation*}
and
\begin{equation*}
\begin{split}
\big(\alpha \wedge \big(d^{\nabla} \wedge (A \wedge A)\big)\big)(X_1,X_2,X_3)&=\frac{1}{2}\sum_{\sigma \in S_3} sign(\sigma) \alpha(X_{\sigma(1)}) d^{\nabla} A \big(A(X_{\sigma (2)}), A(X_{\sigma (3)})\big)\\
&=\alpha(X_1)d^{\nabla} A(A(X_2),A(X_3))-\alpha(X_2)d^{\nabla} A(A(X_1),A(X_3))\\
&-\alpha(X_3)d^{\nabla} A(A(X_2),A(X_1)).
\end{split}
\end{equation*}
Equation \ref{onethird} then suggest that
\begin{itemize}
\item if $A$ is $\alpha$-special, then it is both $\alpha,$ and quasi-$\alpha$-integrable ($\alpha \wedge d^{\nabla} A=0 \implies I^{\nabla}_{\alpha, A}=0 \mbox{ and }I^{\alpha, \nabla}_A=0$); 
\item conversely, if $A$ is both $\alpha$-integrable and quasi-$\alpha$-integrable, then it is $\alpha$-special, which can be seen from writing \[I^{\nabla}_{\alpha, A}=\frac{1}{3} I^{\alpha, \nabla}_A-\frac{2}{3} \alpha \wedge d^{\nabla} A\] (i.e.\ $I^{\nabla}_{\alpha, A}=0 \mbox{ and }I^{\alpha, \nabla}_A=0 \implies \alpha \wedge d^{\nabla} A=0$).
\end{itemize}

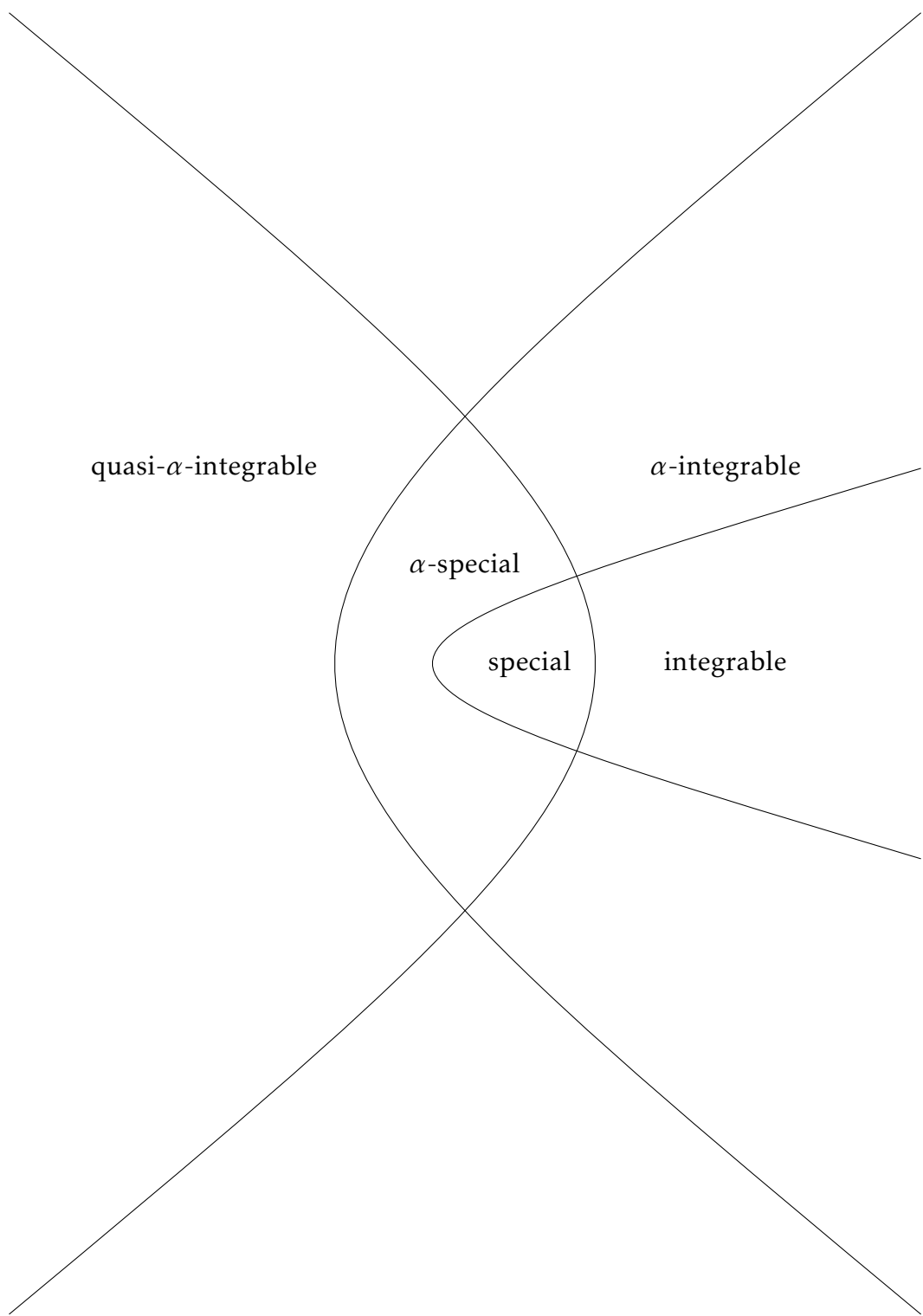
\begin{figure}
\begin{tikzpicture}
\node at (-4,3) {quasi-$\alpha$-integrable};
\node at (4,3) {$\alpha$-integrable};
\node at (4,0) {integrable};
\node at (0,1.5) {$\alpha$-special};
\node at (1,0) {special};
\draw (7,-10) .. controls (-5,0) .. (7,10);
\draw (7,-3) .. controls (-3,0) .. (7,3);
\draw (-7,-10) .. controls (5,0) .. (-7,10);
\end{tikzpicture}
\caption{Relative to a torsion-free connection $\nabla$ on $T_M,$ and a fixed auxiliary $\alpha \in \Omega^1(M)$.}\label{PR1}
\end{figure}

Let $g$ be a Riemannian metric on $M,$ and assume that $\nabla$ is the Levi-Civita connection. Let \[C(M):=\{A \in AC(M) \mid I^{\nabla}_A=0\}\] be the space of complex structures on $M,$ \[AC(M)_g:=\{J \in AC(M) \mid g(J\zeta,J\eta)=g(\zeta,\eta) \forall \zeta,\eta \in \mathfrak{X}(M)\}\] be the space of $g$-orthogonal almost-complex structures, and $C(M)_g:=C(M) \cap AC(M)_g$ be the subspace of integrable such structures. Observe that $A \in AC(M)_g$ is special ($d^{\nabla} A=0$) iff it is K{\"a}hler (i.e.\ iff the almost-hermitian manifold $(M, A, g)$ is K{\"a}hler). And recall the Riemannian geometric defining condition for being K{\"a}hler, which is $\nabla A=0.$ It therefore makes sense to name the $\alpha$-special structures in this setting $\alpha$-K{\"a}hler. More on special almost-complex structures can be found in \cite{ACOTI}, and the sources cited therein. Going back to the equivalence between special and K{\"a}hler, note that indeed, for any $1$-form $A \in \Omega^1(M,T_M),$ $\nabla A=0$ implies that $d^{\nabla} A=0$: \[d^{\nabla} A(X,Y)=(\nabla_X A)Y-(\nabla_Y A)X=0.\] To see why any special $A \in AC(M)_g$ must be K{\"a}hler, let $h(X,Y,Z):=g((\nabla_X A)Y,Z).$ By the compatibility of the Levi-Civita connection, $(\nabla g)(X,Y,Z)=0,$ since $A \in AC(M)_g,$  
\begin{equation*}
\begin{split}
h(X,Y,Z)&=g(\nabla_X (A(Y))-A(\nabla_X Y),Z)\\
&=g(\nabla_X (A(Y)),Z)+g(\nabla_X Y,A(Z))\\
&=\big(\nabla_X g(A(Y),Z)-g(A(Y),\nabla_X Z)\big)+\big(\nabla_X g(Y,A(Z))-g(Y,\nabla_X (A(Z)))\big)\\
&=-g(\nabla_X (A(Z))-A(\nabla_X Z),Y)\\
&=-g((\nabla_X A)Z,Y)\\
&=-h(X,Z,Y).
\end{split}
\end{equation*}
But since at the same time, $(\nabla_X A)Y=(\nabla_Y A)X$ because $A$ is special, $h(X,Y,Z)=h(Y,X,Z),$ so $h(X,Y,Z)=-h(X,Y,Z).$ Then, $h=0,$ so $(\nabla_X A) Y=0$ for all $X,Y\in \mathfrak{X}(M),$ and so $\nabla A=0.$

\begin{figure}
\begin{tikzpicture}
\node at (0,6) {$AC(M)_g$};
\node at (4,1) {$C(M)_g$};
\node at (4,-1) {$C(M)\backslash C(M)_g$};
\node at (0,-6) {$AC(M)\backslash AC(M)_g$};
\node at (4,3) {$\alpha$-integrable};
\node at (-4,3) {quasi-$\alpha$-integrable};
\node at (0,1.5) {$\alpha$-K{\"a}hler};
\node at (0,-1.5) {$\alpha$-special};
\node at (1,0.5) {K{\"a}hler};
\node at (1,-0.5) {special};
\draw (-7,0) -- (7,0);
\draw (7,-10) .. controls (-5,0) .. (7,10);
\draw (7,-3) .. controls (-3,0) .. (7,3);
\draw (-7,-10) .. controls (5,0) .. (-7,10);
\end{tikzpicture}
\caption{Relative to a Riemannian metric $g$ on $M$ with Levi-Civita connection $\nabla$, and a fixed auxiliary $\alpha \in \Omega^1(M)$.}\label{PR2}
\end{figure}
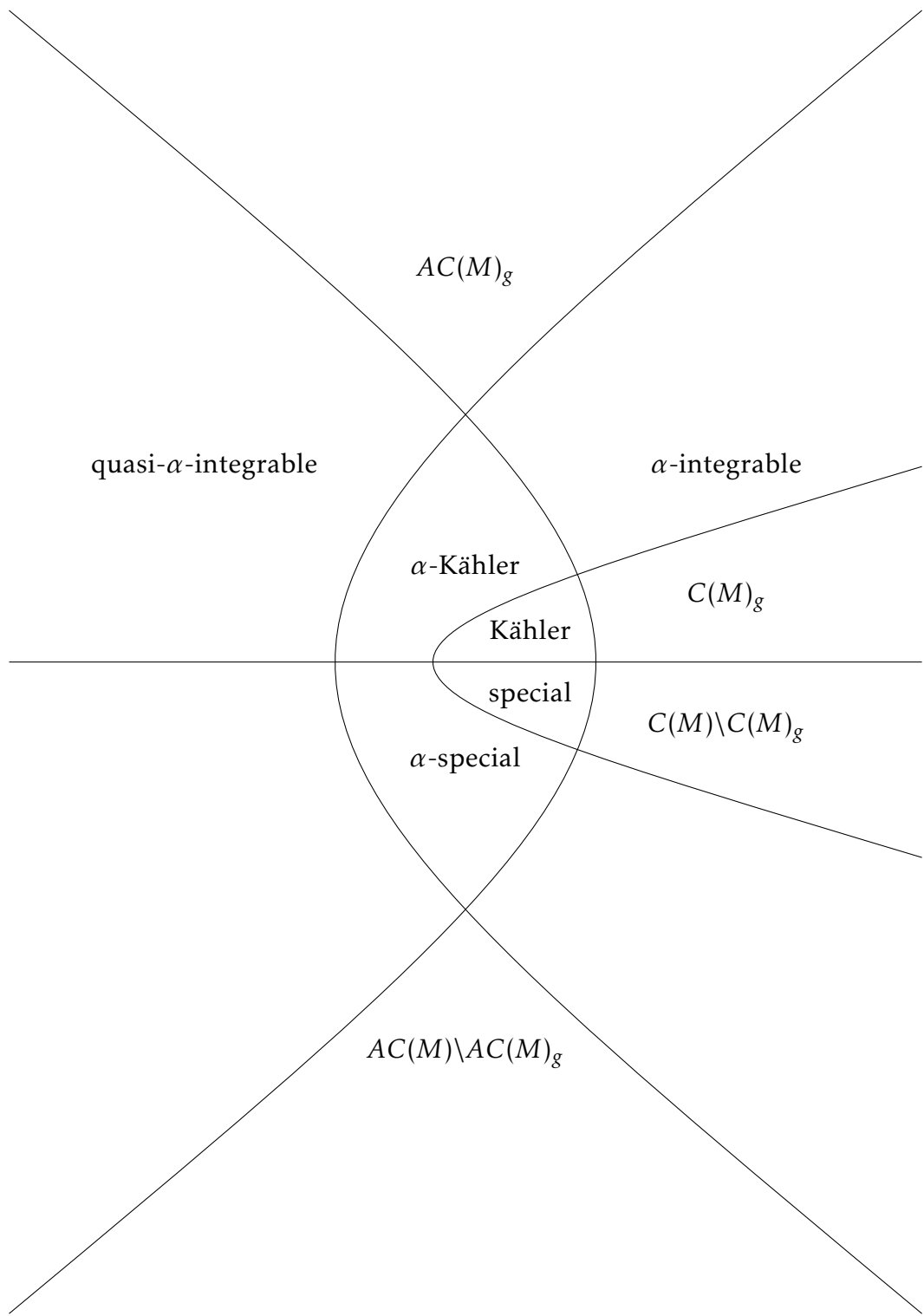

The objective now is to variationally realize all of these almost-complex structures, most importantly, the integrable ones. The problem of finding a functional whose critical points are precisely the complex structures has been an ongoing pursuit. There are examples of functionals for which the complex structures are critical points \cite{Bryant, CLE}, but the issue is usually that there are non-integrable critical points as well. The functionals introduced here bypass this issue. 

Remember the definition of $k$\emph{-variational object} \cite{GSVO1}. This is the degree $k$ projection of any critical point of any functional with domain contained in some space of differential forms, taking values in a vector bundle.

\begin{theorem}\label{EP}
Let $\nabla$ be any torsion-free connection on $T_M,$ and $\alpha \in \Omega^1(M)$ be $d$-closed and different from the zero form. For any $\gamma \in \Omega^{\bullet} (M, T_M),$ let $I^{\alpha, \nabla}_{\gamma}:=\sum_{k\geq 0} I^{\alpha, \nabla}_{\gamma_k},$ and define a functional $\mathcal{C}^{\alpha, \nabla}:\Omega^{\bullet} (M, T_M) \to \mathbb{R}$ by \[\mathcal{C}^{\alpha, \nabla}(\gamma)=\big\langle \big\langle I^{\alpha, \nabla}_{\gamma}, \gamma\big\rangle \big\rangle.\] The quasi-$\alpha$-integrable almost-complex structures on $M$ are $1$-variational objects for $\mathcal{C}^{\nabla,\alpha}.$ 
\end{theorem}

\begin{proof}
The idea is to find an intermediary domain of restriction for $\mathcal{C}^{\alpha, \nabla},$ \[\Omega^{\bullet} (M, T_M) \supset \widetilde{\Omega}^{1}(M, T_M) \supset \widetilde{AC}(M):=\{\gamma \in \widetilde{\Omega}^{1}(M, T_M) \mid \gamma_1 \in AC(M)\},\] such that if $c.p\big(\mathcal{C}^{\alpha, \nabla}\big|_{\widetilde{AC}(M)}\big)$ denotes the set of critical points of $\mathcal{C}^{\alpha, \nabla}\Big|_{\widetilde{AC}(M)},$ then \[p_1 \Big(c.p\big(\mathcal{C}^{\alpha, \nabla}\big|_{\widetilde{AC}(M)}\big)\Big)=\{\mbox{quasi-}\alpha-\mbox{integrable almost-complex structures}\}.\] After computing the first variation of $\mathcal{C}^{\alpha, \nabla},$ the intermediary domain, $\widetilde{\Omega}^{1}(M, T_M),$ can be found by eliminating (1) all terms in the Euler-Lagrange (EL) derivative of $\mathcal{C}^{\alpha, \nabla}$ that depend on degree $1$ components $(\gamma_1),$ except for those that appear in the quasi-$\alpha$-integrability form; and (2) all degree $3$ parts of the EL derivative besides the quasi-$\alpha$-integrability form. These steps are meant to ensure that no extra constraints are placed on the projected critical points of $\mathcal{C}^{\alpha, \nabla}\big|_{\widetilde{AC}(M)}.$ 

First, notice that \[I^{\alpha, \nabla}_{\gamma_k}= \begin{cases} 
      (\alpha \wedge d^{\nabla} \gamma_k) \wedge (\gamma_k \wedge \gamma_k)+d^{\nabla} (\alpha \wedge \gamma_k), &k\mbox{ is odd} \\
    
      d^{\nabla} (\alpha \wedge \gamma_k), & k\mbox{ is even},
   \end{cases}
\]
where $ (\alpha \wedge d^{\nabla} \gamma_k) \wedge (\gamma_k \wedge \gamma_k) \in \Omega^{3k} (M, T_M),$ $d^{\nabla} (\alpha \wedge \gamma_k) \in \Omega^{k+2} (M, T_M).$ So one can write
\begin{equation*}
\begin{split}
I^{\alpha, \nabla}_{\gamma} &=\sum_{k\geq 0} I^{\alpha, \nabla}_{\gamma_{2k+1}} + \sum_{k\geq 0} I^{\alpha, \nabla}_{\gamma_{2k}}\\
&=\sum_{k\geq 0} \big[(\alpha \wedge d^{\nabla} \gamma_{2k+1}) \wedge (\gamma_{2k+1} \wedge \gamma_{2k+1})+d^{\nabla}(\alpha \wedge \gamma_{2k+1})+\\
&d^{\nabla}(\alpha \wedge \gamma_{2k})\big].
\end{split}
\end{equation*}

Thus, 
\begin{equation*}
\begin{split}
\mathcal{C}^{\alpha, \nabla}(\gamma)&=\sum_{k\geq 0} \big[\big\langle (\alpha \wedge d^{\nabla} \gamma_{2k+1}) \wedge (\gamma_{2k+1} \wedge \gamma_{2k+1}), \gamma_{6k+3}\big\rangle_{6k+3}+\\
&\big\langle \alpha \wedge \gamma_{2k+1}, \delta^{\nabla} \gamma_{2k+3}\big\rangle_{2k+2}+\big\langle \alpha \wedge \gamma_{2k}, \delta^{\nabla} \gamma_{2k+2}\big\rangle_{2k+1}\big].
\end{split}
\end{equation*}

The first variation is then
\begin{equation*}
\begin{split}
\frac{d}{dt} \Big|_{t=0} \mathcal{C}^{\alpha, \nabla}(\gamma+t\beta)&=\sum_{k\geq 0} \big[\big\langle(\alpha \wedge d^{\nabla} \beta_{2k+1}) \wedge (\gamma_{2k+1} \wedge \gamma_{2k+1})+\\
&2 (\alpha \wedge d^{\nabla} \gamma_{2k+1}) \wedge (\beta_{2k+1} \wedge \gamma_{2k+1}), \gamma_{6k+3}\big\rangle_{6k+3} +\big\langle \alpha \wedge \beta_{2k+1}, \delta^{\nabla} \gamma_{2k+3}\big\rangle_{2k+2}+\\
&\big\langle \alpha \wedge \beta_{2k}, \delta^{\nabla} \gamma_{2k+2}\big\rangle_{2k+1}\big]+\big\langle \big\langle \beta, I^{\alpha, \nabla}_{\gamma}\big\rangle \big\rangle.
\end{split}
\end{equation*}

Consider the linear maps 
\[L^{\alpha}_{\gamma_{2k+1}}:\Omega^{2k+2} (M, T_M) \to \Omega^{6k+3} (M, T_M), \quad L^{\alpha}_{\gamma_{2k+1}}(B)=(\alpha \wedge B) \wedge (\gamma_{2k+1} \wedge \gamma_{2k+1}),\]

\[M^{\alpha}_{\gamma_{2k+1}}:\Omega^{2k+1} (M, T_M) \to \Omega^{6k+3} (M, T_M), \quad M^{\alpha}_{\gamma_{2k+1}}(b)=2(\alpha \wedge d^{\nabla} \gamma_{2k+1}) \wedge (b \wedge \gamma_{2k+1}),\] and 

\[N^{\alpha}_l:\Omega^l (M, T_M) \to \Omega^{l+1} (M, T_M), \quad N^{\alpha}_l(x)=\alpha \wedge x.\]

Then,
\begin{equation*}
\begin{split}
\frac{d}{dt} \Big|_{t=0} \mathcal{C}^{\alpha, \nabla}(\gamma+t\beta)&=\sum_{k\geq 0} \big[\big\langle \beta_{2k+1}, \delta^{\nabla} \big((L^{\alpha}_{\gamma_{2k+1}})^*(\gamma_{6k+3})\big)+(M^{\alpha}_{\gamma_{2k+1}})^* (\gamma_{6k+3})+\\
&(N^{\alpha}_{2k+1})^*(\delta^{\nabla} \gamma_{2k+3})\big\rangle_{2k+1}+\big\langle \beta_{2k}, (N^{\alpha}_{2k})^* (\delta^{\nabla} \gamma_{2k+2})\big\rangle_{2k}\big]+ \big\langle \big\langle \beta, I^{\alpha, \nabla}_{\gamma}\big\rangle \big\rangle.
\end{split}
\end{equation*}

Now, since the residual terms involving $\gamma_1$ (see (1) at the beginning of the proof) are $\delta^{\nabla} \big((L^{\alpha}_{\gamma_1})^*(\gamma_3)\big),$ and $(M^{\alpha}_{\gamma_1})^* (\gamma_3),$ take $\gamma_3=0.$ And since the degree $3$ part of the EL equation is \[\delta^{\nabla} \big((L^{\alpha}_{\gamma_3})^*(\gamma_9)\big)+(M^{\alpha}_{\gamma_3})^* (\gamma_9)+(N^{\alpha}_3)^*(\delta^{\nabla} \gamma_5)+I^{\alpha, \nabla}_{\gamma_1}=0,\] take $\gamma_9=0,$ $\delta^{\nabla} \gamma_5=0$ (see (2)). So, let the intermediary domain of restriction be given as \[\widetilde{\Omega}^1(M,T_M):=\{\gamma \in \Omega^{\bullet} (M, T_M) \mid \gamma_3=0, \gamma_9=0, \delta^{\nabla} \gamma_5=0\},\] and recall that \[\widetilde{AC}(M)=\{\gamma \in \widetilde{\Omega}^1(M,T_M) \mid \gamma_1 \in AC(M)\}.\] If $\mathcal{S}$ is the set of critical points of $\mathcal{C}^{\alpha, \nabla},$ the critical point set of $\mathcal{C}^{\alpha, \nabla}\big|_{\widetilde{AC}(M)}$ is $\mathcal{S} \cap \widetilde{AC}(M),$ and $p_1\big(\mathcal{S} \cap \widetilde{AC}(M)\big)$ coincides with the quasi-$\alpha$-integrable almost-complex structures.
\end{proof}
The above proof has the following immediate consequences. But first, recall from \cite{GSVO1} the adjusted covariant exterior derivative that vanishes in prescribed degrees. Denote the restriction $d^{\nabla}\Big|_{\Omega^k (M, T_M)}$ by $d_k^{\nabla}:\Omega^k (M, T_M) \to \Omega^{k+1} (M, T_M),$ and likewise let $\delta^{\nabla}\Big|_{\Omega^k (M, T_M)}$ be denoted by $\delta_k^{\nabla}:\Omega^k (M, T_M) \to \Omega^{k-1} (M, T_M).$ For any $k,$ define an operator on $\Omega^{\bullet} (M, T_M),$ \[d^{\nabla}[k]:=\sum_{i=0}^{n-1}(1-\delta_{k-1}^i)d_i^{\nabla}\] so that 
\[d^{\nabla}[k]\Big|_{\Omega^i (M, T_M)}= \begin{cases} 
      d_i^{\nabla}, & i \neq k-1 \\
      0 , & i = k-1.
   \end{cases}
\]

This has a formal adjoint $\delta^{\nabla}[k],$ where  
\[\delta^{\nabla}[k]\Big|_{\Omega^i (M, T_M)}= \begin{cases} 
      \delta_i^{\nabla}, & i \neq k \\
      0 , & i = k.
   \end{cases}
\]

\begin{corollary}\label{red}
Redefine $I^{\alpha, \nabla}_{\gamma}$ using $d^{\nabla}[5],$ namely $I^{\alpha, \nabla}_{\gamma}[5]:=\sum_{k\geq 0} I^{\alpha, \nabla}_{\gamma_k}[5],$ where $I^{\alpha, \nabla}_{\gamma_k}[5]=(\alpha \wedge d^{\nabla}[5] \gamma_k) \wedge (\gamma_k \wedge \gamma_k)-\alpha \wedge  d^{\nabla}[5] \gamma_k.$ The quasi-$\alpha$-integrable almost-complex structures are $1$-variational for $\mathcal{C}^{\alpha, \nabla}[5]:\Omega^{\bullet} (M, T_M) \to \mathbb{R},$ where \[\mathcal{C}^{\alpha, \nabla}[5](\gamma)=\big\langle \big\langle I^{\alpha, \nabla}_{\gamma}[5], \gamma\big\rangle \big\rangle.\] 
\end{corollary}

The intermediary domain of restriction now has the more relaxed definition \[\widetilde{\Omega}^1(M,T_M):=\{\gamma \in \Omega^{\bullet} (M, T_M) \mid \gamma_3=0, \gamma_9=0\}.\] 

\begin{corollary}\label{flow}
Let \[(L^{\alpha}_{\gamma})^* (\gamma)=\sum_{k\geq 0} (L^{\alpha}_{\gamma_{2k+1}})^*(\gamma_{6k+3}),\] \[(M^{\alpha}_\gamma)^*(\gamma)=\sum_{k\geq 0} (M^{\alpha}_{\gamma_{2k+1}})^* (\gamma_{6k+3}),\] and \[(N^{\alpha})^*(\delta^{\nabla} \gamma)=\sum_{k\geq 0} (N^{\alpha}_k)^*(\delta^{\nabla} \gamma_{k+2}).\] Formally, the gradient flow of $\mathcal{C}^{\alpha, \nabla}$ is 
\begin{equation*}
\begin{split}
\frac{\partial \gamma}{\partial t}&=-\Big(\delta^{\nabla} \big((L^{\alpha}_{\gamma})^* (\gamma)\big)+(M^{\alpha}_\gamma)^*(\gamma)+(N^{\alpha})^*(\delta^{\nabla} \gamma)+I^{\alpha, \nabla}_{\gamma}\Big).
\end{split}
\end{equation*}
\end{corollary}

The modification $\mathcal{C}^{\alpha, \nabla}[5]$ has a nearly identical evolution equation (replace every instance of $d^{\nabla}$ with $d^{\nabla}[5]$). 

Both integrable and $\alpha$-integrable structures can be classified in essentially the same manner as the quasi-$\alpha$-integrable ones. But the proofs are worked out here in detail due to some subtle differences, plus it seems like a good idea to be comprehensive when the way of thinking is not so standard. 

There is also a desire to emphasize a pattern that might not be entirely specific to almost-complex manifolds, and which could be meaningful for studying, more broadly, the (non-)existence of common structures in differential geometry.

\begin{theorem}\label{allnew}
Let $I^{\nabla}_{\gamma_k}[1]:=d^{\nabla}[1]\gamma_k \wedge (\gamma_k \wedge \gamma_k)-d^{\nabla}[1] \gamma_k,$ and $I^{\nabla}_{\gamma}[1]=\sum_{k\geq 0} I^{\nabla}_{\gamma_k}[1].$ Complex structures are $1$-variational objects, which can be realized via the functional $\mathcal{C}^{\nabla}[1]:\Omega^{\bullet} (M, T_M) \to \mathbb{R},$ \[\mathcal{C}^{\nabla}[1](\gamma)=\big\langle \big\langle I^{\nabla}_{\gamma}[1], \gamma\big\rangle \big\rangle.\] 
\end{theorem}

\begin{proof}
The proof is in the style of Theorem \ref{EP}'s. The first variation is

\begin{equation*}
\begin{split}
\frac{d}{dt} \Big|_{t=0} \mathcal{C}^{\nabla}[1](\gamma+t\beta)&=\sum_{k\geq 0} \big[\big\langle d^{\nabla}[1] \beta_{2k+1} \wedge (\gamma_{2k+1} \wedge \gamma_{2k+1})+\\
&2 d^{\nabla}[1] \gamma_{2k+1} \wedge (\beta_{2k+1} \wedge \gamma_{2k+1}), \gamma_{6k+2}\big\rangle_{6k+2} -\big\langle \beta_{2k+1}, \delta^{\nabla}[1] \gamma_{2k+2}\big\rangle_{2k+1}\\
&-\big\langle \beta_{2k}, \delta^{\nabla}[1] \gamma_{2k+1}\big\rangle_{2k}\big]+\big\langle \big\langle \beta, I^{\nabla}_{\gamma}[1] \big\rangle \big\rangle.
\end{split}
\end{equation*}
Consider the linear maps 
\[L_{\gamma_{2k+1}}:\Omega^{2k+2} (M, T_M) \to \Omega^{6k+2} (M, T_M), \quad L_{\gamma_{2k+1}}(B)=B \wedge (\gamma_{2k+1} \wedge \gamma_{2k+1}),\] and

\[M_{\gamma_{2k+1}}:\Omega^{2k+1} (M, T_M) \to \Omega^{6k+2} (M, T_M), \quad M_{\gamma_{2k+1}}(b)=2d^{\nabla}[1] \gamma_{2k+1} \wedge (b \wedge \gamma_{2k+1}).\] 

Then, 
\begin{equation*}
\begin{split}
\frac{d}{dt} \Big|_{t=0} \mathcal{C}^{\nabla}[1](\gamma+t\beta)&=\sum_{k\geq 0} \big[\big\langle \beta_{2k+1}, \delta^{\nabla}[1] \big((L_{\gamma_{2k+1}})^*(\gamma_{6k+2})\big)+(M_{\gamma_{2k+1}})^* (\gamma_{6k+2})\\
&-\delta^{\nabla}[1] \gamma_{2k+2}\big\rangle_{2k+1}-\big\langle \beta_{2k}, \delta^{\nabla}[1] \gamma_{2k+1}\big\rangle_{2k}\big]+ \big\langle \big\langle \beta, I^{\nabla}_{\gamma}[1]\big\rangle \big\rangle.
\end{split}
\end{equation*}

In order to eliminate occurrences of $\gamma_1$ happening outside of the intergability form, $\delta^{\nabla}[1] \big((L_{\gamma_1})^*(\gamma_2)\big)+(M_{\gamma_1})^* (\gamma_2)=0,$ take $\gamma_2=0.$ Notice how the co-closed condition on $\gamma_1$ is neutralized with the choice of operator $d^{\nabla}[1]$ (in place of $d^{\nabla}$). This is a subtle issue that does not come up in the proof of Theorem \ref{EP}. Additionally, since the degree $2$ part of the EL equation is $\delta^{\nabla}[1] \gamma_3+I^{\nabla}_{\gamma_1}[1]=\delta^{\nabla} \gamma_3+I^{\nabla}_{\gamma_1}=0,$ require $\delta^{\nabla} \gamma_3=0.$ The intermediary domain should then be \[\widetilde{\Omega}^1(M,T_M):=\{\gamma \in \Omega^{\bullet} (M, T_M) \mid \gamma_2=0, \delta^{\nabla}\gamma_3=0\}.\] Letting $\mathcal{S}$ stand for the critical point set of the functional, $\mathcal{S} \cap \widetilde{AC}(M)$ are the critical points of its restriction to $\widetilde{AC}(M)=\{\gamma \in \widetilde{\Omega}^1(M,T_M) \mid \gamma_1 \in AC(M)\},$ and \[p_1\big(\mathcal{S} \cap \widetilde{AC}(M)\big)=\{\mbox{ complex structures }\}\] as desired.
\end{proof}

\begin{corollary}\label{redy}
Redefining $I^{\nabla}_{\gamma}[1]$ using $d^{\nabla}[1, 3]:=\sum_{i=0}^{n-1}(1-\delta_0^i-\delta_2^i)d_i^{\nabla},$ (cf.\ Corollary \ref{red}, and the brief discussion following the proof of Theorem \ref{EP}), one finds that the integrable almost-complex structures are $1$-variational for $\mathcal{C}^{\nabla}[1, 3]:\Omega^{\bullet} (M, T_M) \to \mathbb{R},$ where \[\mathcal{C}^{\nabla}[1, 3](\gamma)=\big\langle \big\langle I^{\nabla}_{\gamma}[1, 3], \gamma\big\rangle \big\rangle.\] 
\end{corollary}

The intermediary domain of restriction now becomes \[\widetilde{\Omega}^1(M,T_M):=\{\gamma \in \Omega^{\bullet} (M, T_M) \mid \gamma_2=0\}.\] 

\begin{corollary}\label{flows}
Let \[(L_{\gamma})^* (\gamma)=\sum_{k\geq 0} (L_{\gamma_{2k+1}})^*(\gamma_{6k+2}),\] and \[(M_\gamma)^*(\gamma)=\sum_{k\geq 0} (M_{\gamma_{2k+1}})^* (\gamma_{6k+2}).\] The formal gradient flow of $\mathcal{C}^{\nabla}[1]$ is 
\begin{equation*}
\begin{split}
\frac{\partial \gamma}{\partial t}&=-\Big(\delta^{\nabla}[1] \big((L_{\gamma})^* (\gamma)\big)+(M_\gamma)^*(\gamma)-\delta^{\nabla}[1] \gamma+I^{\nabla}_{\gamma}[1] \Big).
\end{split}
\end{equation*}
\end{corollary}

Similar remarks about flows apply to $\mathcal{C}^{\nabla}[1, 3].$ Finally,

\begin{theorem}\label{allnew2}
The $\alpha$-integrable almost-complex structures are $1$-variational objects. They can be realized via the functional $\mathcal{C}^{\nabla}_{\alpha}:\Omega^{\bullet} (M, T_M) \to \mathbb{R},$ \[\mathcal{C}^{\nabla}_{\alpha}(\gamma)=\big\langle \big\langle I^{\nabla}_{\alpha, \gamma}, \gamma\big\rangle \big\rangle,\] where $I^{\nabla}_{\alpha, \gamma}=\sum_{k\geq 0} \alpha \wedge I^{\nabla}_{\gamma_k},$ and $I^{\nabla}_{\gamma_k}:=d^{\nabla} \gamma_k \wedge (\gamma_k \wedge \gamma_k)-d^{\nabla} \gamma_k.$
\end{theorem}

\begin{proof}
Following the notation used in the preceding proofs, put $L_{\alpha,\gamma_{2k+1}}:=\alpha \wedge L_{\gamma_{2k+1}},$ and $M_{\alpha,\gamma_{2k+1}}:=\alpha \wedge M_{\gamma_{2k+1}}.$ Then,

\begin{equation*}
\begin{split}
\frac{d}{dt} \Big|_{t=0} \mathcal{C}^{\nabla}_{\alpha}(\gamma+t\beta)&=\sum_{k\geq 0} \big[\big\langle L_{\alpha,\gamma_{2k+1}}(d^{\nabla} \beta_{2k+1})+M_{\alpha,\gamma_{2k+1}}( \beta_{2k+1}),\gamma_{6k+3}\big\rangle_{6k+3}+\\
&\big\langle N^{\alpha}_{2k+1}(\beta_{2k+1}), \delta^{\nabla} \gamma_{2k+3}\big\rangle_{2k+2}+\big\langle N^{\alpha}_{2k}(\beta_{2k}), \delta^{\nabla} \gamma_{2k+2}\big\rangle_{2k+1}\big]+\\
&\big\langle \big\langle I^{\nabla}_{\alpha, \gamma}, \beta\big\rangle \big\rangle\\
&=\sum_{k\geq 0} \big[\big\langle \beta_{2k+1},\delta^{\nabla} \big((L_{\alpha, \gamma_{2k+1}})^*(\gamma_{6k+3})\big)+(M_{\alpha, \gamma_{2k+1}})^* (\gamma_{6k+3})+\\
&(N^{\alpha}_{2k+1})^*(\delta^{\nabla} \gamma_{2k+3})\big\rangle_{2k+1}+\big\langle \beta_{2k}, (N^{\alpha}_{2k})^*(\delta^{\nabla} \gamma_{2k+2})\big\rangle_{2k}\big]+\big\langle \big\langle \beta, I^{\nabla}_{\alpha, \gamma} \big\rangle \big\rangle.
\end{split}
\end{equation*}

With the intermediary domain \[\widetilde{\Omega}^1(M,T_M):=\{\gamma \in \Omega^{\bullet} (M, T_M) \mid \gamma_3=0, \gamma_9=0, \delta^{\nabla} \gamma_5=0\},\] $\widetilde{AC}(M)$ defined as before, and letting $\mathcal{S}$ be the critical point set of $\mathcal{C}^{\nabla}_{\alpha},$ it follows that \[p_1\big(\mathcal{S} \cap \widetilde{AC}(M)\big)=\{\alpha-\mbox{integrable almost-complex structures}\}\] as desired. 
\end{proof}

As in the case of quasi-$\alpha$-integrable structures (Corollary \ref{red}), adjust $\mathcal{C}^{\nabla}_{\alpha}$ by redefining $I^{\nabla}_{\alpha, \gamma}$ with $d^{\nabla}[5].$ Then,

\begin{corollary}\label{redys}
The $\alpha$-integrable almost-complex structures are $1$-variational for the modified functional $\mathcal{C}^{\nabla}_{\alpha}[5].$ 
\end{corollary}

This alteration gets rid of the co-closed condition in degree $5,$ simplifying the intermediary domain to, once again, \[\widetilde{\Omega}^1(M,T_M):=\{\gamma \in \Omega^{\bullet} (M, T_M) \mid \gamma_3=0, \gamma_9=0\}.\]

\begin{corollary}\label{flowss}
Let \[(L_{\alpha,\gamma})^* (\gamma)=\sum_{k \geq 0} (L_{\alpha, \gamma_{2k+1}})^* (\gamma_{6k+3}),\] \[(M_{\alpha,\gamma})^*(\gamma)=\sum_{k \geq 0}(M_{\alpha,\gamma_{2k+1}})^*(\gamma_{6k+3}),\] and \[(N^{\alpha})^*(\delta^{\nabla} \gamma)=\sum_{k\geq 0} (N^{\alpha}_k)^*(\delta^{\nabla} \gamma_{k+2}).\] The formal gradient flow of $\mathcal{C}^{\nabla}_{\alpha}$ is 
\begin{equation*}
\begin{split}
\frac{\partial \gamma}{\partial t}&=-\Big(\delta^{\nabla} \big((L_{\alpha,\gamma})^* (\gamma)\big)+(M_{\alpha,\gamma})^*(\gamma)+(N^{\alpha})^*(\delta^{\nabla} \gamma)+I^{\nabla}_{\alpha, \gamma} \Big).
\end{split}
\end{equation*}
\end{corollary}

And, as expected, $\mathcal{C}^{\nabla}_{\alpha}[5]$ has its own gradient flow too.

\section{The existence problem}
To understand when a given manifold admits geometric structures of a specific kind, it makes sense to observe the evolution of these structures as the manifold degenerates, in contextually relevant directions, to its worst possible condition. For example, if the setting is K{\"a}hler geometry, and the geometric structures of interest are K{\"a}hler-Einstein metrics on Fano manifolds, then the principle governing the existence of these metrics is called K-stability \cite{CDS}. Here, the contextually relevant directions in the degeneration process are known as test-configurations, and the degenerated end state is the central fiber of the test-configuration. The evolution of canonical K{\"a}hler metrics is inferred from a rational number, the Donaldson-Futaki (DF) invariant. K-stability is the requirement that this invariant stay positive over all non-trivial test-configurations. 

More explicitly, let $X$ be a Fano manifold, and $\Omega$ be a fixed K{\"a}hler class. The K-energy is the unique functional $\mathcal{K}$ on $\Omega$ whose critical points are the constant scalar curvature K{\"a}hler metrics, and such that $\mathcal{K}(0)=0.$ Let $L \rightarrow X$ be an ample line bundle; i.e.\ for some $r,$ there is a basis of sections in $H^0(X,L^r),$ giving an embedding $\epsilon:X \hookrightarrow \mathbb{CP}^{N_r},$ $\epsilon(p)=[s_0(p):\dots:s_{N_r}(p)].$ A \emph{test-configuration} is the data $(\epsilon, \lambda),$ where $\lambda$ is a $\mathbb{C}^*$-action on $\mathbb{CP}^{N_r},$ or equivalently, $\lambda:\mathbb{C}^* \hookrightarrow \GL_{N_r +1}(\mathbb{C})$ is a 1-parameter subgroup. The \emph{central fiber} is the flat limit $X_0:=\lim_{t \to 0} \lambda(t) \cdot X.$ For all $t \neq 0,$ $\omega_t:=\frac{1}{t} \epsilon^* (\omega_{FS})$ is a K{\"a}hler metric on $\lambda(t) \cdot X.$ Up to normalization, the DF invariant of the test-configuration $(\epsilon,\lambda)$ is the asymptotic derivative \[DF(\epsilon,\lambda):=\lim_{t \to \infty} \frac{d\mathcal{K}(\omega_t)}{dt}.\] The concept from the beginning of the section can be phrased in the (almost-)complex setting by analogy with K-stability. 

Choose a reference structure $J \in AC(M).$ As shown in \cite{DG}, and generalized in \cite{GUES}, there is an embedding $F$ of $(M,J)$ into a complex manifold $(Z,J_Z)$ that is transverse to a co-rank $n$ distribution $D \subset T_Z.$ In fact, $F(M)$ is contained in a suitably defined real part $Z^{\mathbb{R}}$ of $Z.$ By transversality, there is an $\mathbb{R}$-linear isomorphism $g:F_* (T_M) \to {T_Z}/D.$ Since $J_Z$ descends to a complex structure on ${T_Z}/D,$ $J^{Z,D}_M:=g^{-1} \circ J_Z \circ g$ is an almost-complex structure on $F(M) \subset Z^{\mathbb{R}}.$ 

Now, let $\theta$ be a $\GL^{+}_1(\mathbb{R})$-action on $Z^{\mathbb{R}},$ where $\GL^{+}_1(\mathbb{R})$ is the identity component. Call the pair $(F,\theta)_J$ a \emph{test-configuration of} M \emph{relative (rel.)} $J.$ Set $M_t:=\theta(t) \cdot M.$ For all $t>0,$ the map $f_t:M \to M_t,$ $f_t(p)=\theta(t)p,$ is a diffeomorphism, and so $J_t:=(f_t)_{*} \circ J^{Z,D}_M \circ (f_t)^{-1}_{*}$ is an almost-complex structure on $M_t.$ 

Pick an extension $\gamma_{J_t} \in \widetilde{AC}(M_t).$ The functionals from the previous section give rise to seemingly different stability criteria. Say that $M$ is\\

\noindent
\emph{quasi-}$\alpha$-C\emph{-stable rel.} $J$ iff for all test-configurations $(F,\theta)_J$ rel.\ $J,$ \[C^{\alpha, \nabla} \big((F,\theta)_J\big):=\lim_{t \to \infty} \frac{d\mathcal{C}^{\alpha, \nabla}(\gamma_{J_t})}{dt} \geq 0;\]

\noindent 
$\alpha$-C\emph{-stable rel.} $J$ iff for all test-configurations $(F,\theta)_J$ rel.\ $J,$ \[C^{\nabla}_{\alpha} \big((F,\theta)_J\big):=\lim_{t \to \infty} \frac{d\mathcal{C}^{\nabla}_{\alpha}(\gamma_{J_t})}{dt} \geq 0; \mbox{ and}\] 

\noindent
C\emph{-stable rel.} $J$ iff for all test-configurations $(F,\theta)_J$ rel.\ $J,$ \[C^{\alpha, \nabla} \big((F,\theta)_J\big):=\lim_{t \to \infty} \frac{d\mathcal{C}^{\nabla}(\gamma_{J_t})}{dt} \geq 0.\]

And then say that $M$ is \emph{quasi-}$\alpha$-C\emph{-stable}, $\alpha$-C\emph{-stable}, and C-\emph{stable} iff it is \emph{quasi-}$\alpha$-C\emph{-stable rel.} $J,$ $\alpha$-C\emph{-stable rel.} $J,$ and respectively C\emph{-stable rel.} $J$ for all $J  \in AC(M).$ Similar stability concepts ensue from the functionals in Corollaries \ref{red}, \ref{redy}, and \ref{redys}.

The geometric content of these definitions is yet to be understood. In particular, it is not clear what the central fiber $M_0$ represents, and if the degenerations are appropriate for studying the CSE problem. But provided that these classes of stable almost-complex manifolds are non-empty, it could be interesting to try to understand what do they say about the (almost-)complex geometry of $M.$ In particular, one may ask:
\begin{question}
Are complex manifolds stable in any of the above senses?
\end{question}

\noindent
Gabriella Clemente                                

\noindent
e-mail: clemente6171@gmail.com
\end{document}